\documentclass[12pt]{article}
\usepackage[margin=1in]{geometry}
\usepackage{amsfonts,amsmath,amssymb, amsthm}
\usepackage{xcolor,colortbl}
\usepackage[notcite, notref,color,final]{showkeys}
\usepackage[flushleft]{threeparttable}
\usepackage{hyperref}
\definecolor{refkey}{rgb}{0,0,1}
\definecolor{labelkey}{rgb}{0,0,1}
\usepackage{fancyvrb}

\newcommand{\mage}{\color{magenta}}
\newcommand{\myatop}[2]{\genfrac{}{}{0pt}{}{#1}{#2}}
\usepackage{lineno}
\usepackage{xspace}
\author{Damanvir Singh Binner\thanks{Indian Institute of Science Education and
	Research (IISER), Mohali, India.  damanvirbinnar@iisermohali.ac.in } 
 \and Amarpreet Rattan\thanks{Department of Mathematics, Simon Fraser
University, Burnaby, Canada.  rattan@sfu.ca}}
\newtheorem{theorem}{Theorem}

\newtheorem{lemma}[theorem]{Lemma}

\newtheorem{prop}[theorem]{Proposition}

\theoremstyle{remark}
\newtheorem*{remark}{Remark}

\makeatletter
\newcommand{\aone}{\@ifstar{{\mage (A1)}\xspace}{{\mage A1}\xspace}}
\newcommand{\atwo}{\@ifstar{{\mage (A2)}\xspace}{{\mage A2}\xspace}}
\newcommand{\athree}{\@ifstar{{\mage (A3)}\xspace}{{\mage A3}\xspace}}
\newcommand{\afour}{\@ifstar{{\mage (A4)}\xspace}{{\mage A4}\xspace}}
\newcommand{\afive}{\@ifstar{{\mage (A5)}\xspace}{{\mage A5}\xspace}}
\newcommand{\asix}{\@ifstar{{\mage (A6)}\xspace}{{\mage A6}\xspace}}
\newcommand{\bnewone}{\@ifstar{{\mage (B1)}\xspace}{{\mage B1}\xspace}}
\newcommand{\bone}{\@ifstar{{\mage (B2)}\xspace}{{\mage B2}\xspace}}
\newcommand{\btwo}{\@ifstar{{\mage (B3)}\xspace}{{\mage B3}\xspace}}
\newcommand{\bthree}{\@ifstar{{\mage (B4)}\xspace}{{\mage B4}\xspace}}
\newcommand{\bfive}{\@ifstar{{\mage (B5)}\xspace}{{\mage B5}\xspace}}
\newcommand{\cone}{\@ifstar{{\mage (C1)}\xspace}{{\mage C1}\xspace}}
\newcommand{\done}{\@ifstar{{\mage (D1)}\xspace}{{\mage D1}\xspace}}
\newcommand{\dtwo}{\@ifstar{{\mage (D2)}\xspace}{{\mage D2}\xspace}}
\newcommand{\dthree}{\@ifstar{{\mage (D3)}\xspace}{{\mage D3}\xspace}}
\newcommand{\eone}{\@ifstar{{\mage (E1)}\xspace}{{\mage E1}\xspace}}
\newcommand{\etwo}{\@ifstar{{\mage (E2)}\xspace}{{\mage E2}\xspace}}
\newcommand{\fone}{\@ifstar{{\mage (F1)}\xspace}{{\mage F1}\xspace}}
\newcommand{\ftwo}{\@ifstar{{\mage (F2)}\xspace}{{\mage F2}\xspace}}
\newcommand{\gone}{\@ifstar{{\mage (G1)}\xspace}{{\mage G1}\xspace}}
\newcommand{\gtwo}{\@ifstar{{\mage (G2)}\xspace}{{\mage G2}\xspace}}
\newcommand{\gthree}{\@ifstar{{\mage (G3)}\xspace}{{\mage G3}\xspace}}
\newcommand{\gfour}{\@ifstar{{\mage (G4)}\xspace}{{\mage G4}\xspace}}
\newcommand{\hone}{\@ifstar{{\mage (H1)}\xspace}{{\mage H1}\xspace}}
\newcommand{\ione}{\@ifstar{{\mage (I1)}\xspace}{{\mage I1}\xspace}}
\newcommand{\itwo}{\@ifstar{{\mage (I2)}\xspace}{{\mage I2}\xspace}}
\newcommand{\jone}{\@ifstar{{\mage (J1)}\xspace}{{\mage J1}\xspace}}
\newcommand{\kone}{\@ifstar{{\mage (K1)}\xspace}{{\mage K1}\xspace}}
\makeatother

\newcommand{\freq}{f}
\begin{document}

\title{\Large{A Comparison of Integer Partitions Based on Smallest Part}}

\date{}
\maketitle

\begin{abstract}
	For positive integers $n, L$ and $s$, consider the following two sets that both contain partitions of $n$ with
	the difference between the largest and smallest parts bounded by $L$:  the
	first set contains partitions with smallest part
	$s$, while the second set contains partitions with smallest part at least $s+1$.  Let $G_{L,s}(q)$
	be the generating series whose coefficient of
	$q^n$ is difference between the sizes of the above two sets of partitions.  This
	generating series was introduced by Berkovich and Uncu in 2019.   Previous
	results concentrated on the nonnegativity of $G_{L,s}(q)$ in the cases $s=1$
	and $s=2$.  In the present paper, we show the eventual positivity of $G_{L,s}(q)$
	for general $s$ and also find a precise nonnegativity result for the case
	$s=3$.  
 \end{abstract}
 
\section{Introduction} \label{Intro}

Let $n$ be a nonnegative integer. A \emph{partition $\pi = (\pi_1,
\pi_2, \dots)$ of $n$} is a weakly decreasing list of positive
integers whose sum is $n$, and we write $|\pi| = n$ to indicate this.  We
allow the empty partition as the unique partition of $0$.  Each $\pi_i$
is known as a \emph{part} of $\pi$.  A standard way to visualize $\pi$ is through its \emph{Ferrers diagram}, which is a collection of left justified  rows of boxes, with the $i^\mathrm{th}$ row containing $\pi_i$ boxes.

In the present article, it is
more convenient to use the notation that expresses the number of parts
of each size in a partition.   In this notation, we write $\pi = (1^{f_1}, 2^{f_2},
\ldots)$, where $f_i$ is the \emph{frequency} of $i$ or the number of times a
part $i$ occurs in $\pi$.   Thus, each frequency $f_i$ is a nonnegative integer,
and if $f_i = 0$, then $\pi$ has no part of size $i$.  When the frequency
of a number is 0, it may or may not be omitted in the expression.  In the latter
notation, it is clear that $|\pi| = \sum_i i \cdot f_i$.   Thus $(4,4,2,2,1)$,
$(1^1, 2^2, 3^0, 4^2, 6^0)$, and $(1^1, 2^2, 4^2, 5^0)$ all represent the same
partition of 13.  We let $s(\pi)$ and $l(\pi)$ denote the smallest and largest parts of $\pi$,
    respectively, and $\mho$ denotes the set of partitions $\pi$ with $|\pi| > 0$.

%
For indeterminates $a$ and $q$, and a positive integer $n$, define
$$
(a; q)_n := (1-a)(1-aq) \cdots (1-aq^{n-1}).
$$
The central objects of this article are the following series.  For positive
integers $L$, $s$ and $k$, define
\begin{itemize}
\item $G_{L,s}(q)$ to be the generating series
    \begin{equation}
        G_{L,s}(q) := \sum_{\substack{\pi \in \mho, \\ s(\pi)=s, \\ l(\pi)-s(\pi)
        \leq L}} q^{|\pi|} -  \sum_{\substack{\pi \in \mho, \\ s(\pi) \geq s+1, \\
        l(\pi)-s(\pi) \leq L}} q^{|\pi|},\label{GLsq}
    \end{equation} and 
\item $H_{L,s,k}(q)$ to be the generating series 
    \begin{equation}\label{eq:hlsk}
        H_{L,s,k}(q) := \frac{q^s(1-q^k)}{(q^s;q)_{L+1}} - \left(\frac{1}{(q^{s+1};q)_L}-1\right).
    \end{equation}
\end{itemize}

A series $\sum_{n \geq 0} a_n q^n$ is said to be \emph{eventually
positive} if there exists some $l \in \mathbb{N}$ such that $a_n > 0$ for all $n \geq l$.  

Berkovich and Uncu conjectured that $H_{L,s,k}(q)$ is a eventually positive for
all positive integers triples $L \geq 3$, $s$ and $k \geq s+1$.  This conjecture was recently proved independently by Zang
and Zeng \cite[Theorem 1.3]{zang} and by the present authors \cite[Section
3]{BR20}.  The two proofs are substantially different.  While the proof of Zang and Zeng was partly combinatorial and partly
analytic, the proof of the present authors was entirely combinatorial. The
present authors in fact proved a stronger result. To state this, we need the following notation:

\begin{align}\label{eq:defplsgammals}
    \begin{split}
        &\bullet\; P_{L,s} = (s+1)(s+2)\ldots(s+L),\\
        &\bullet\; \gamma(L,s) =
        \left(\left(s+1\right)+\left(s+2\right)+\cdots+\left(s+L\right)\right)\\
        &\hspace{6.2cm} \cdot
		\left(P_{L,s}^{\left(P_{L,s}^2-1\right)L+2}+\left(\left(P_{L,s}^2-1\right)L-2\right)P_{L,s}\right),\\
        &\bullet\;  \Gamma(s) = \gamma(3s+2, s).
    \end{split}
    \end{align}
	The following theorem appears in \cite[Theorem 5]{BR20}.
\begin{theorem}
\label{Genk}
  For positive integers $L$, $s$ and $k$, with $L \geq 3$ and $k \geq s+1$, the coefficient of
  $q^N$ in $H_{L,s,k}(q)$ is positive whenever $N \geq \Gamma(s)$.
  \end{theorem}
  We emphasize that Theorem \ref{Genk} is stronger than the conjecture of
  Berkovich and Uncu about $H_{L,s,k}(q)$ as it explicitly gives the bound for
  when $H_{L,s,k}(q)$ is positive, and it further states this bound depends only on $s$. 

Berkovich and Uncu \cite[Theorems 5.1 and
5.2]{BerkovichAlexander2017SEPI} showed that
the series $G_{L,s}(q)$ and $H_{L,s,k}(q)$, for $s=1$ and $s=2$ and $L \geq 1$,
satisfy the following simple relationship: 
\begin{equation}\label{s=1}
	G_{L,s}(q) = \frac{H_{L,s,L}(q)}{1-q^L}.
\end{equation}

A series $S(q) = \sum_{n \geq 0} a_nq^n $ is said to be \emph{nonnegative} if $a_n
\geq 0$ for all $n$. The nonnegativity of the series $S(q)$ is denoted by $S(q) \succeq 0.$  

Berkovich and Uncu \cite[Theorem 5.1]{BerkovichAlexander2017SEPI} prove, using
\eqref{s=1}, that $G_{L,1}(q) \succeq 0$.  Also using \eqref{s=1}, they conjectured
Theorem \ref{s=2} below, which pertains to the nonnegativity of $G_{L,2}(q)$.
Theorem \ref{s=2} was proved by the present authors \cite[Section 3]{BR20}.

\begin{theorem}
\label{s=2}
For $L=3$, 
\begin{equation*}
	G_{L,2}(q)+q^3 + q^9 + q^{15} \succeq 0;
\end{equation*}
for $L=4$, 
\begin{equation*}
	G_{L,2}(q)+q^3 + q^9 \succeq 0;
\end{equation*}
and for $L \geq 5$,
\begin{equation*}
	G_{L,2}(q)+q^3  \succeq 0.
\end{equation*}
\end{theorem}


In the present article, we explore the nonnegativity properties of $G_{L,s}(q)$ for
general positive integers $s$.  We begin the study of this series with the next result. 
\begin{theorem}
\label{Generals}
For positive integers $L \geq 1$, $$ G_{L,s}(q) = \frac{H_{L,s,L}(q)}{1-q^L}. $$
\end{theorem}
We prove Theorem \ref{Generals} in Section \ref{PGLs}, but we emphasize that the
proof is essentially the one given by Berkovich and Uncu for the cases $s=1$
and $s=2$, with only minor modifications.  We include this proof for completeness.  Next we show that for any $L \geq
s+1$, the series $G_{L,s}(q)$ is eventually positive, and the bound after which
the coefficient of $q^N$ is positive can be written explicitly in terms of $s$ only.  Define the quantities
 \begin{equation}
       \label{deltas}
\begin{aligned}
 &\delta(s) := e^{3\Gamma(s)},\\
\textnormal{ and }\;\;\; &\delta'(s) := 10s+(s+2)(s+3)(\delta(s)+1).
\end{aligned}
\end{equation}
Obviously $\delta'(s) > \delta(s)$ for all positive $s$.  
\begin{theorem}
 \label{GLs}
    If $s$ and $L \geq s+1$ are positive integers, then the coefficient of $q^n$
	in $G_{L,s}(q)$ is positive whenever $n \geq \delta'(s)$, so $G_{L,s}(q)$
	eventually positive.
\end{theorem}  
We prove Theorem \ref{GLs} in Section \ref{PGLs}.  Then we focus on the case
$s=3$ and obtain an extension of Theorem \ref{s=2};  that is, we show that, with the
exception of a few small terms, the series $G_{L,s}(q)$ is nonnegative.  The
next result states this precisely, and its proof is in Section \ref{PGLthree}.  
\begin{theorem}
\label{GLthree} 
For $L \geq 10$, $$ G_{L,3}(q) +q^4+q^5+q^8+q^{10}+q^{12}+q^{14}+q^{16} \succeq 0.$$ For $5 \leq L \leq 9$, we have the following results.
\begin{gather*}
 G_{9,3}(q)+q^4+q^5+q^8+q^{10}+q^{12}+q^{14}+2q^{16} \succeq 0. \\
 G_{8,3}(q)+q^4+q^5+q^8+q^{10}+q^{12}+q^{14}+2q^{16}+q^{20} \succeq 0. \\
 G_{7,3}(q)+q^4+q^5+q^8+q^{10}+q^{12}+2q^{14}+q^{16}+q^{20} \succeq 0. \\
 G_{6,3}(q)+q^4+q^5+q^8+q^{10}+q^{12}+q^{13}+2q^{14}+2q^{16}+q^{18}+2q^{20}+q^{22} \succeq 0. \\
 G_{5,3}(q)+q^4+q^5+q^8+q^{10}+2q^{12}+q^{13}+q^{14}+2q^{16}
 +q^{17}+q^{18}+3q^{20}\phantom{+q^{22}+q^{24}}\\
 \phantom{G_{5,3}(q)+q^4+q^5+q^8+q^{10}+2q^{12}+q^{13}xxxxxxxxxxxxxxxx} +q^{22}+q^{24}+q^{28} \succeq 0. 
 \end{gather*}
and for $L=4$, 
\begin{multline*}
G_{4,3}(q)+ q^4 + q^5 + q^8 + q^{10} + q^{11} + 2q^{12} + 2q^{14} + 3q^{16} + q^{17} \\
 + 2q^{18} + q^{19} + 4q^{20} + 3q^{22} + q^{23} +  4q^{24} + q^{25} + 4q^{26} + 5q^{28} \\
  + q^{29} + 3q^{30} + 6q^{32} + 3q^{34} + 4q^{36} + 2q^{38} + 4q^{40} + 2q^{44}     \succeq 0. 
\end{multline*}
\end{theorem}
We point out that the bound in Theorem \ref{GLs} is likely far from optimal.  Take, for
example, the case $s=3$.  According to Theorem \ref{GLs}, the coefficient of $q^n$ in
$G_{L,3}(q)$ is nonnegative whenever $n \geq \delta'(3)$, where $\delta'(3)$ is
is extremely large.  However, from
Theorem \ref{GLthree}, the coefficient of $q^n$ in $G_{L, 3}(q)$ is
nonnegative whenever $n \geq 45$.   This suggests that the bound in Theorem
\ref{GLs} can be improved greatly. 

Theorems \ref{Generals}, \ref{GLs} and \ref{GLthree} can also be found in the
PhD. thesis of the first author \cite{binner}.

\section{Proofs of Theorems \ref{Generals} and \ref{GLs}}
\label{PGLs}

We begin by proving Theorem \ref{Generals}. 

\begin{proof}[Proof of Theorem \ref{Generals}]

The definition of $G_{L,s}(q)$ in \eqref{GLsq} is given as the difference of 
two generating series.  We begin by finding a rational expression for the first 
generating series.  All the partitions counted by this generating series have 
smallest part equal to $s$ and largest part at most $L+s$. 
Hence we obtain for the first generating series the expression
\begin{equation}
\label{8}
 \sum_{\substack{\pi \in \mho, \\ s(\pi)=s, \\ l(\pi)-s(\pi)  \leq L}} q^{|\pi|} = \frac{q^s}{(1-q^s)(1-q^{s+1}) \cdots (1-q^{L+s})}  = \frac{q^s}{(q^s;q)_{L+1}}.
 \end{equation}
  
 For the second generating series in the definition of $G_{L,s}(q)$, we fix the
  number of parts of the partition to be $n$ and then sum over all $n$. Suppose
  $\pi$ is a partition with $n$ parts, where each part is at least $s+1$. Then, in the Ferrers diagram of $\pi$, the whole column over the smallest part of $\pi$ is generated by the $q$-factor $$ \frac{q^{(s+1)n}}{1-q^n}. $$
Stripping the columns above the smallest part from the far left of the Ferrers 
diagram of $\pi$, we are left with a new partition that has at most $n-1$ parts 
and largest part bounded above by $L$.  It is well known 
(see for example \cite[Proposition 1.1]{Aigner}) that these partitions 
are generated by the $q$-binomial coefficient 
$$\left[\myatop{L+n-1}{n-1} \right]_q :=  
\frac{(q;q)_{L+n-1}}{(q;q)_L (q;q)_{n-1}}.$$ 
Thus, for the second generating series in the definition of $G_{L,s}(q)$, we have 
$$ \sum_{\substack{\pi \in \mho, \\ s(\pi) \geq s+1, \\ l(\pi)-s(\pi) \leq L}} 
q^{|\pi|} =  \sum_{n=1}^{\infty}\frac{q^{(s+1)n}}{1-q^n}
\left[\myatop{L+n-1}{n-1} \right]_q.$$ 
Simplifying the summands on the right hand side, we find 
\begin{align*}
	\frac{1}{1-q^n} \left[\myatop{L+n-1}{n-1} \right]_q  &= \frac{1}{1-q^L}
	\left[\myatop{L+n-1}{n} \right]_q \\
 &= \frac{1}{1-q^L} \frac{(q^L;q)_n}{(q;q)_n}.
 \end{align*}
 Therefore
 \begin{align}
  \sum_{\substack{\pi \in \mho, \\ s(\pi) \geq s+1, \\ l(\pi)-s(\pi) \leq L}} q^{|\pi|} &= 
\frac{1}{1-q^L}  \sum_{n=1}^{\infty} q^{(s+1)n}  \frac{(q^L;q)_n}{(q;q)_n}\notag\\
  &=  \frac{1}{1-q^L} \left( -1 + \sum_{n=0}^{\infty} q^{(s+1)n}  
\frac{(q^L;q)_n}{(q;q)_n}  \right)\notag\\
  &=  \frac{1}{1-q^L} \left( \frac{1}{(q^{s+1};q)_L} - 1 \right),\label{16} 
  \end{align}
  where the last step follows from the $q$-binomial theorem (see 
\cite[$(2.1)$]{BerkovichAlexander2017SEPI}) 
$$\sum_{n=0}^{\infty} 
\frac{(a;q)_n}{(q;q)_n} z^n = \frac{(az;q)_{\infty}}{(z;q)_{\infty}}$$ 
with $a=q^L$ and $z=q^{s+1}$.
  Substituting \eqref{8} and \eqref{16} into the definition of $G_{L,s}(q)$ gives us
  \begin{align*}
  G_{L,s}(q) &= \frac{q^s}{(q^s;q)_{L+1}} -  \frac{1}{1-q^L} \left( \frac{1}{(q^{s+1};q)_L} - 1 \right) \\
  &= \frac{1}{1-q^L} \left( \frac{q^s (1-q^L)}{(q^s;q)_{L+1}} -  \left( \frac{1}{(q^{s+1};q)_L} - 1 \right)  \right) \\
  &=  \frac{1}{1-q^L} H_{L,s,L}(q),
  \end{align*}
  as required.
\end{proof}

Theorem \ref{Generals} expresses $G_{L,s}(q)$ in terms of $H_{L,s,L}(q)$, while 
Theorem \ref{Genk} gives the explicit bound $\Gamma(s)$ after which coefficients in the 
series $H_{L,s,L}(q)$ are positive.  We use these to show nonnegativity
properties of $G_{L,s}(q)$.  We first show in Theorem \ref{thm:proofconj5} 
that there is a bound $M$, which depends only on $L$ and $s$, such that the coefficient 
of $q^n$ in $G_{L,s}(q)$ is nonnegative whenever $n \geq M$.

For positive integers $s$ and $L \geq s+1$, let $H_{L,s,L}(q) 
= \sum_{n \geq 0} a_{L,n}q^n$ and $G_{L,s}(q) = \sum_{n \geq 0} b_{L,n}q^n$. Then 
Theorem \ref{Generals} implies
\begin{equation}
\label{HtoG}
 b_{L,n} = a_{L,n} + a_{L,n-L} + a_{L,n-2L} + \cdots = \sum_{\substack{m \leq n \\ m \equiv n (\text{mod} L)}} a_{L,m}.
 \end{equation}
 We introduce some more notation:
\begin{itemize}
\item $\eta_1(L,s) = \sum_{n < \Gamma(s)} |a_{L,n}|$;
\item $\eta_2(L,s) = \max(\eta_1(L,s), \Gamma(s))$; and
\item $\eta_3(L,s) = (L+1) \eta_2(L,s)$.
\end{itemize} 
\begin{theorem}
 \label{thm:proofconj5}
    Let $s$ and $L \geq s+1$ be positive integers.  Then the coefficient of $q^n$ 
in $G_{L,s}(q)$ is nonnegative whenever $n \geq \eta_3(L,s)$. 
        \end{theorem}  

\begin{proof}
Suppose $n \geq \eta_3(L,s)$. We can rewrite \eqref{HtoG} as  
\begin{equation}
\label{rewrite}
b_{L,n} = \sum_{\substack{\eta_2(L,s) \leq m \leq n \\ m \equiv n (\text{mod} L)}} a_{L,m} +  \sum_{\substack{\Gamma(s) \leq m < \eta_2(L,s) \\ m \equiv n (\text{mod} L)}} a_{L,m} + \sum_{\substack{m < \Gamma(s) \\ m \equiv n (\text{mod} L)}} a_{L,m}.
\end{equation}
Note that the second sum may be empty.  Since $n \geq \eta_3(L, s)$, the first sum on the 
right hand side of \eqref{rewrite} contains at least $\eta_2(L,s)$ terms, all of which 
are positive by Theorem \ref{Genk}. Thus
\begin{equation}
\label{first3}
 \sum_{\substack{\eta_2(L,s) \leq m \leq n \\ m \equiv n (\text{mod} L)}} a_{L,m}  \geq \eta_2(L,s). 
 \end{equation}
For the second sum in the right hand side of \eqref{rewrite}, Theorem \ref{Genk}
gives
\begin{equation}
\label{second2}
 \sum_{\substack{\Gamma(s) \leq m < \eta_2(L,s) \\ m \equiv n (\text{mod} L)}} a_{L,m} \geq 0.
\end{equation}
For the third sum in the right hand side of \eqref{rewrite}, using the triangle inequality, we obtain 
$$  \left| \sum_{\substack{m < \Gamma(s) \\ m \equiv n (\text{mod} L)}} a_{L,m}  \right| \leq \sum_{\substack{m < \Gamma(s) \\ m \equiv n (\text{mod} L)}} |a_{L,m}| \leq \sum_{m < \Gamma(s)} |a_{L,m}|  = \eta_1(L,s) \leq \eta_2(L,s),  $$ and thus 
\begin{equation}
\label{third} 
 \sum_{\substack{m < \Gamma(s) \\ m \equiv n (\text{mod} L)}} a_{L,m}  \geq -\eta_2(L,s). 
\end{equation}
The result now follows immediately from \eqref{rewrite}, \eqref{first3}, \eqref{second2} 
and \eqref{third}. 
\end{proof}

The bound $\eta_3(L,s)$ in Theorem \ref{thm:proofconj5} guaranteeing when the coefficients of $G_{L,s}(q)$ are nonnegative depends on both $L$ and $s$.   To prove Theorem \ref{GLs}, we need to 
find a bound that only depends on $s$, and this is our next goal.   We again use
the connection between $G_{L,s}(q)$ and $H_{L,s,L}(q)$ in Theorem
\ref{Generals};  while Theorem
\ref{Genk} guarantees the series $H_{L,s,L}(q)$ is eventually positive,
we also need a lower bound on the size of the coefficients of $H_{L,s,L}(q)$.  This is the content of Theorem \ref{ModiGenk} below, a
strengthening of Theorem \ref{Genk} in the case $k=L$. 

Let
\begin{itemize}
	\item $D_{L,s}$  denote the set of nonempty partitions
		with parts in the set $\{s+1,  \ldots, L+s\}$, and  
	\item $I_{L,s,L}$ be the set of partitions where the smallest part is $s$, all parts are 
$\leq L+s$, and $L$ does not appear as a part.
\end{itemize}  
From the definition of the series $H_{L,s,L}(q)$ in \eqref{eq:hlsk}, when $L
\geq s+1$, elementary partition theory gives the coefficient of $q^N$ in $H_{L,s,L}(q)$ as
the difference
\begin{equation}\label{eq:comb}
 |\{\pi \in I_{L,s,L} : |\pi| = N|\}|  - |\{\pi \in D_{L,s} : |\pi| = N\}|.
\end{equation}
Thus to prove nonnegativity of the coefficient of $q^N$ in $H_{L,s,L}(q)$, it 
suffices to show there exists an injection $\phi$ such that
\begin{equation}\label{eq:gamma}
	\phi: \{\pi \in D_{L,s} : |\pi| = N\} \rightarrow \{\pi \in I_{L,s,L} : |\pi|
	= N\}
\end{equation}
Equations \eqref{eq:comb} and \eqref{eq:gamma} are central to proving our
theorems below in this section and the next.  We also need the following result.
\begin{prop}\label{thm:prop1}
    For given positive integers $a,b$ and $n$ with $\gcd(a,b)=1$, the 
number 
of solutions of $ax+by=n$ in nonnegative integer pairs $(x,y)$ is either $\lfloor 
\frac{n}{ab} \rfloor$ or $\lfloor \frac{n}{ab} \rfloor + 1$.
\end{prop}
See \cite{AT}.   See also \cite[Chapter 5]{niven} for an elementary proof.

 \begin{theorem}
\label{ModiGenk}
  For positive integers $L \geq 3$ and $s$, with $L \geq s+1$, the coefficient of
  $q^N$ in $H_{L,s,L}(q)$ is greater than or equal to $\left \lfloor \frac{N-10s}{(s+2)(s+3)} \right \rfloor$ whenever $N \geq \Gamma(s)$.  
   \end{theorem} 
  \begin{proof}
	  The proof of Theorem \ref{Genk}, found in \cite{BR20}, uses the
	  combinatorial interpretation of the coefficients of $H_{L,s,L}(q)$ in
	  \eqref{eq:comb}; the proof there constructs for $N \geq \Gamma(s)$ an injection
	  as in \eqref{eq:gamma} to show nonnegativity of \eqref{eq:comb}.  To show positivity of
\eqref{eq:comb}, elements of the codomain that are not in the range of $\phi$
are given.  The details of these injections and elements are dealt with in different cases and theorems of \cite{BR20} depending
on the relative sizes of $L$ and $s$.\footnote{As indicated by Theorem \ref{Genk}, the proof in \cite{BR20} for the positivity of $H_{L,s,k}(q)$ applies for all $k \geq
	s+1$.  In particular, a more general $I_{L,s,k}$ is defined than the one
given in \eqref{eq:comb}.  Here we are only interested in the case $k=L$.} For example, when $L \geq
2s+3$, there is no partition of the form $(s^{10}, (s+1)^x, (s+2)^y)$ in the
range of $\phi$, but such partitions are in the codomain.   To achieve the
present result, we count the number of such partitions;  this will then give the
desired lower bound for the difference \eqref{eq:comb} in each case.   We then
compare the results of all cases.

Table \ref{tab:casesbound} lists the cases, the theorem from \cite{BR20}
showing positivity of \eqref{eq:comb}, the partitions in the codomain but
not the range of $\phi$ showing positivity of \eqref{eq:comb}, and the
enumeration
of these partitions.
\begin{table}[htbp]
	\centering
	\caption{The first column describes the case, the second the theorem from
		\cite{BR20} that proves positivity in this case, the third the
		partitions in the codomain that are not in the range of $\phi$ in
		\eqref{eq:gamma}, and the last column contains the number of partitions
	of the type in Column 3}.  
	\label{tab:casesbound}
	\begin{tabular}{|c|c|c|c|c|}
		\hline
		& Case of $L, s$ &  Theorem & Partitions in codomain & min.
		num,\\
		& & of \cite{BR20}   & not in range of $\phi$ & of partitions\\
		\hline
		1 & $L \geq 2s+3$ & 10 & $(s^{10},(s+1)^x,(s+2)^y)$ & \rule{0pt}{18pt} $\left\lfloor
		\frac{N-10s}{(s+1)(s+2)} \right\rfloor$ \\
		[+0.5em]
		\hline
			2 & $s+3 \leq L \leq 2s+2$ & 12 & $(s^{1},(s+1)^x,(s+2)^y)$ & \rule{0pt}{18pt} $\left \lfloor
		\frac{N-s}{(s+1)(s+2)} \right \rfloor$\\
		[+0.5em]
		\hline
		3 & $L = s+1$ & 12 & $(s^{1},(s+2)^x,(s+3)^y)$ &\rule{0pt}{18pt} $\left \lfloor
		\frac{N-s}{(s+2)(s+3)} \right \rfloor$\\
		[+0.5em]
		\hline 
		4 & $s$ even,  $L=s+2$  & 12 & $(s^{1},(s+1)^x,(s+3)^y)$ & \rule{0pt}{18pt} $\left \lfloor
		\frac{N-s}{(s+1)(s+3)} \right \rfloor$\\
		[+0.5em]
		\hline
		5 & $s$ odd,  $L=s+2$, $N$ odd & 12 &  $(s^{1},(s+1)^x,(s+3)^y)$ & \rule{0pt}{18pt} $\left
	\lfloor \frac{2(N-s)}{(s+1)(s+3)} \right \rfloor$\\
		[+0.5em]
	\hline
		6 & $s$ odd,  $L=s+2$, $N$ even, $s \neq 1$ & 12 & $(s^{2},(s+1)^x,(s+3)^y)$ &
		\rule{0pt}{18pt} $\left \lfloor \frac{2(N-2s)}{(s+1)(s+3)} \right \rfloor$\\
		[+0.5em]
	\hline
		7 & $s =1$,  $L=3$, $N$ even & 12 & $(1^{6},2^x,4^y)$ &
		\rule{0pt}{18pt} $\left \lfloor \frac{(N-6)}{4} \right \rfloor$\\
	[+0.5em]
	 \hline
	\end{tabular}
\end{table}
To count the number of partitions in Column 3 in Table \ref{tab:casesbound}, in
Rows 1-4 
a straight forward application of Proposition \ref{thm:prop1} to $n = N - ts$, where $t$ is the
number of parts of $s$, and $a$ and $b$ set to be the remaining parts, which are
coprime, gives the numbers in Column 4.

The remaining rows require a slightly more work.  For Row 5, the numbers $s+1$
and $s+3$ have greatest common factor 2, so we can apply Proposition
\ref{thm:prop1} to 
$n=\tfrac{N-s}{2}, a=\tfrac{s+1}{2}$ and $b=\tfrac{s+3}{2}$.   The count in
Column 4 then follows.  The analysis for Rows 6 and 7 are similar.

The result is now obtained by observing that all the values in Column 4 exceed $\left \lfloor \frac{N-10s}{(s+2)(s+3)} \right \rfloor$.

  \end{proof}

For a positive integer $m$, let $p(m)$
be the number of partitions of $m$.  We need the following result of de Azevedo
Pribitkin \cite{Prib}. 
\begin{theorem}\label{thm:prop2}
    Let $m$ be a positive integer.  Then $p(m) \leq e^{3\sqrt{m}}$.  
\end{theorem}

\begin{proof}[Proof of Theorem \ref{GLs}]

	As before, let $a_{L,n}$ and $b_{L,n}$ be the coefficient of $q^n$ in
	$H_{L,s,L}(q)$ and $G_{L,s}(q)$, respectively.
Further, recall the definitions of $\Gamma(s), \delta(s),$ and $\delta'(s)$ in
	\eqref{eq:defplsgammals} and $\eqref{deltas}$.

Suppose $n \geq \delta'(s)$.  Again from \eqref{HtoG}, we have 
\begin{equation}
\label{rewrite'}
b_{L,n} = \sum_{\substack{\delta'(s) \leq m \leq n \\ m \equiv n (\text{mod}
L)}} a_{L,m} +  \sum_{\substack{\Gamma(s) \leq m < \delta'(s) \\ m \equiv n
(\text{mod} L)}} a_{L,m} + \sum_{\substack{m < \Gamma(s) \\ m \equiv n
(\text{mod} L)}} a_{L,m}.
\end{equation}
For $m \geq \delta'(s)$, we have $m \geq \Gamma(s)$ and $\lfloor
\frac{m-10s}{(s+2)(s+3)} \rfloor \geq \delta(s)$, so $a_{L,m} \geq \delta(s)$
by Theorem \ref{ModiGenk}.   The first sum in the right hand side of
\eqref{rewrite'} contains at least 1 term (the term $m=n$) and each term in the sum is greater than or equal to $\delta(s)$. Thus
\begin{equation}
\label{first'}
 \sum_{\substack{\delta'(s) \leq m \leq n \\ m \equiv n (\text{mod} L)}} a_{L,m}  \geq \delta(s). 
 \end{equation}
For the second sum in the right hand side of \eqref{rewrite'}, from Theorem
\ref{Genk} it follows
\begin{equation}
\label{second'}
 \sum_{\substack{\Gamma(s) \leq m < \delta'(s) \\ m \equiv n (\text{mod} L)}} a_{L,m}  \geq 0.
\end{equation}
For the third sum on the right hand side of \eqref{rewrite'}, the
combinatorial interpretation of $H_{L,s,L}(q)$ in \eqref{eq:comb} gives $a_{L,m}
\geq - p(m)$ for any $m \in \mathbb{N}$.  By Theorem \ref{thm:prop2}, we then have $$a_{L,m} \geq - p(m) \geq -e^{3\sqrt{m}} \geq -e^{3m}.$$ Therefore
\begin{equation}
\label{third'} 
 \sum_{\substack{m < \Gamma(s) \\ m \equiv n (\text{mod} L)}} a_{L,m}  \geq -\sum_{\substack{m < \Gamma(s) \\ m \equiv n (\text{mod} L)}} e^{3m}  \geq -\sum_{m < \Gamma(s)} e^{3m} > -\delta(s),
\end{equation}
where the last inequality follows from the familiar formula for finite geometric
sums and the definition of $\delta(s)$.
The theorem now follows immediately from \eqref{rewrite'}, \eqref{first'}, \eqref{second'} and \eqref{third'}. 

\end{proof}

\section{Proof of Theorem \ref{GLthree}}
\label{PGLthree}

To prove Theorem \ref{GLthree}, we use the connection between $G_{L,3}(q)$
and $H_{L, 3, L}(q)$ given in Theorem \ref{Generals}.  To use this relationship,
we need to understand the coefficients of $q^N$ in $H_{L, 3, L}(q)$ for
small $N$, so we need a result, in the case $s=3$ and $L \geq s+1$,
stronger than Theorem \ref{Genk}.  Our strategy
for proving the coefficient of $q^N$ in $H_{L,3,L}(q)$ is nonnegative for
small $N$ is to show that the coefficient of $q^N$ in
$H_{L,3,L}(q)$ is nonnegative when $N$ is larger than a small bound; this is
done in Lemmas \ref{Helpful2} - \ref{Five}.  Then we use the lemmas along with
machine computation and Theorem \ref{Generals} to prove Theorem \ref{GLthree} at
the end of the section.

Recall the following fundamental result of Sylvester \cite{Sylvester82}.

\begin{theorem}[Sylvester's theorem]
	Let $a$ and $b$ be positive coprime integers.  Then the equation $ax+by = n$ has a
	nonnegative integer solution $(x,y)$  whenever $n \geq (a-1)(b-1)$.
\end{theorem}

We additionally need the following lemmas.

\begin{lemma}
\label{4,5,6}
Let $n \geq 4$ be a positive integer such that $n \neq 7$. Then the equation $4x+5y+6z=n$ has a solution in nonnegative integer triples $(x,y,z)$. 
\end{lemma}

\begin{lemma}
\label{5,6,7}
Let $n \geq 5$ be a positive integer such that $n \neq 8,9$. Then the equation $5x+6y+7z=n$ has a solution in nonnegative integer triples $(x,y,z)$.
\end{lemma}

\begin{lemma}
\label{4,5,6,7}
Let $n \geq 4$ be a positive integer. Then the equation $4x+5y+6z+7u=n$ has a solution in nonnegative integer tuples $(x,y,z,u)$. 
\end{lemma}

The proofs of these lemmas are all simple applications of Sylvester's theorem.
For example, the proof of Lemma \ref{4,5,6} can be obtained by applying
Sylvester's theorem to $a=4$ and $b=5$, which establishes the conclusion for all
$n \geq 12$, and then each smaller $n$ can be dealt with individually.  The proofs  
of Lemmas \ref{5,6,7} and \ref{4,5,6,7} are similar, so we omit the details.

We also need another lemma.

\begin{lemma}
\label{twosol}
Suppose the equation $4x+5y+6z=n$ has a solution $(\alpha, \beta, \gamma)$ in nonnegative integer triples. Then the equation $4x+5y+6z=n+6$ has a solution different from $(\alpha, \beta, \gamma+1)$ whenever $n \geq 4$ and $n \neq 5$.
\end{lemma}

\begin{proof}
First suppose $\alpha \geq 1$. Then $(\alpha-1, \beta+2, \gamma)$ is a required
solution.  Next suppose $\alpha = 0$.  If $\gamma \geq 1$, then $(\alpha+3, \beta, \gamma-1)$ is a required
solution.  If $\gamma = 0$, then $\beta \geq 2$ because of the restriction on
$n$, and $(\alpha+4, \beta-2, \gamma)$ is a required solution. 
\end{proof}

The next lemma shows that the coefficient of $q^N$ in $H_{L,3,L}(q)$ is positive
for most values of $L$ and small $N$.  Positivity, as opposed to nonnegativity,
is needed in this case to prove Theorem \ref{GLthree} at the end of the
section.

\begin{lemma}
\label{Helpful2}
For $L \geq 22$ and $N \geq 21$, the coefficient of $q^N$ in $H_{L,3,L}(q)$ is
positive.
\end{lemma}

\begin{proof}
Fix $L \geq 22$ and $N \geq 21$.
Recall the combinatorial interpretation of the coefficients of $H_{L,3,L}(q)$ in \eqref{eq:comb}.
We prove nonnegativity of the coefficients of $H_{L,3,L}(q)$ by
constructing an injective map $\phi$ as in
\eqref{eq:gamma} for $s=3$.   Positivity will then
be shown at the end of the proof by displaying an element of the codomain of
$\phi$ that is not in its range.  

Let $\pi = \left(4^{f_4},
\ldots, L^{f_L}, \ldots, (L+3)^{f_{L+3}}\right)$ be a partition of $N$ in $D_{L,3}$ and
let $\freq$ denote $f_L$.  Recall that partitions of $N$ in $I_{L, 3, L}$, the
codomain of $\phi$, have
parts in the set $\{3, \ldots, L+3\}$, the number 3 must occur as a part, and
$L$ does not occur as a part.   Our proof of injectivity of $\phi$ is as
follows.  
\begin{itemize}
	\item We define $\phi(\pi)$ in cases chiefly determined by
$\freq$ in $\pi$, with several subcases that depend on the frequencies of other
parts of $\pi$.  In each case, it will usually be readily apparent that $\phi$
is injective, but we will provide some justification for more complicated cases.
	\item When we analyze why $\phi$ is injective overall, we will gather cases by the frequency of 3 in the image of $\phi$.  Thus each case is
labelled twice: first by its case determined by the frequency $\freq$ of $L$ in $\pi$
(see below, for example, Case 2(c)(ii)($\alpha$)) and second, in parentheses, by
the frequency of $3$ in $\phi(\pi)$ (for example, \bone*).  Once the cases are gathered by their frequency of 3 in the image
of $\phi$, we analyze, for each fixed $i$, all cases where the frequency of 3 is $i$
in the image of $\phi$, and we argue why $\phi$ is injective collectively in
these cases.  For example, on
Page \pageref{page:pageas} the cases ${\mage (\mathrm{A}*)}$ are all the cases where the
frequency of 3 is 1 in the image of $\phi$.  
	\item We then argue that $\phi$ must be
injective overall because distinct cases where in the image of $\phi$ the frequency of 3
in one case is $i$ and the frequency of 3 in the other case is $j$, where $i
\neq j$, cannot contain common elements, so two partitions $\pi$ and $\pi'$
pertaining to distinct cases $i$ and $j$ cannot have the same image.
\end{itemize}

We now define $\phi$.

Case 1 {\mage (}\fone, \kone{\mage )} (this case, exceptionally, has many values
for the frequency of 3 in a partition in the image of $\phi$, so it has more than one pink label): $f\geq 1$.  Notice $(L-18)i \geq 4$ for all $i \geq 1$, so the equation 
\begin{equation}\label{eq:spread}
	(L-18)i = 4x_i + 5y_i +6z_i + 7u_i
\end{equation} has a nonnegative integer solution by Lemma \ref{4,5,6,7}. For each $i \geq 1$, fix such a 
 solution $x_i, y_i, z_i$ and $u_i$. Define $$\phi (\pi) = \left(3^{6f}, 4^{f_4+x_f}, 5^{f_5+y_f}, 6^{f_6+z_f}, 7^{f_7+u_f}, \ldots, L^0, \ldots \right). $$
The function $\phi$ is injective in this case.  Given a partition 
$\phi (\pi) = \left(3^{6f}, 4^{a}, 5^{b}, 6^{c}, 7^{d}, \ldots, L^0, \ldots
\right)$ in the range $\phi$, we can infer $\pi$ comes from this case (no cases below
have the same frequency of 3).  From the frequency of 3 in $\phi(\pi)$, we can infer $\freq$;  then,
from \eqref{eq:spread}, we can infer $x_f, y_f, z_f$ and $u_f$;  finally, from
$\freq$ and $x_f, y_f, z_f$ and $u_f$, we can reconstruct $\pi$.

 Case 2: $f=0$.  We have the following subcases. Recall that the smallest part of $\pi$ is denoted by $s(\pi)$.

 Case 2(a) \bnewone*:   $s(\pi) = L+3$.  Then $\pi = ((L+3)^{\freq_{L+3}})$.  Define
	 $$\phi(\pi) = \left(3^2, 4^2, 5^1, (L-16), \ldots, (L+3)^{\freq_{L+3} - 1}\right).$$
	Note $L-16 \geq 6$ because $L \geq 22$.

 Case 2(b) \aone*: $7 \leq s(\pi) < L+3$.  Define $$\phi(\pi) = \left(3^1,
 (s(\pi)-3)^1, (s(\pi)^{f_{s(\pi)}-1}), \ldots, \right). $$
 Note $s(\pi) - 3 \neq L$, so no part of size $L$ is created. 

 Case 2(c): $s(\pi) \leq 6$. We have the following subcases.
 
 Case 2(c)(i) \cone*: $f_4 \geq 1$ and $f_5 \geq 1$.  Define $$\phi(\pi) = \left(3^3, 4^{f_4-1},5^{f_5-1}, 6^{f_6}, \ldots \right). $$
 
 Case 2(c)(ii): $f_4=0$ or $f_5 =0$.  We have the following subcases.
 
 Case 2(c)(ii)($\alpha$) \bone*: $f_6 \geq 1$. Define $$ \phi(\pi) = (3^2, 4^{f_4}, 5^{f_5}, 6^{f_6-1}, \ldots). $$
 
 Case 2(c)(ii)($\beta$): $f_6=0$. Thus in this subcase either $f_4=f_6=0$ or
 $f_5=f_6=0$, and since $s(\pi) \leq 6$, precisely one of these two conditions holds. We have further subcases.
   
  Case 2(c)(ii)($\beta$)(I): $f_4=f_6=0$. Then $\pi = \left(5^{f_5}, 7^{f_7}, \ldots \right)$. 
  
  Case 2(c)(ii)($\beta$)(I)(A) \eone*: $f_5 \geq 3$. Define $$ \phi(\pi)= (3^5, 5^{f_5-3}, 7^{f_7}, \ldots).$$
  
  Case 2(c)(ii)($\beta$)(I)(B): $f_5 =1$. So $\pi = \left(5^1,7^{f_7}, \ldots,
  \right)$. Let $m_1 \geq 7$  be the least number with a nonzero frequency in
  $\pi$, which must exist because $N \geq 21$. 
    
  Case 2(c)(ii)($\beta$)(I)(B)(i) \atwo*: $m_1 \neq 7, 11, 12$. Then $m_1-3 \geq
  5$ and $m_1-3 \neq 8,9$. By Lemma \ref{5,6,7} there exist some nonnegative integers $u_{m_1-3}, v_{m_1-3}$ and $w_{m_1-3}$ such that 
  \begin{equation}\label{eq:defatwo}
	  m_1-3 = 5 u_{m_1-3} + 6 v_{m_1-3} + 7 w_{m_1-3}. 
\end{equation}
Define $$\phi(\pi) = \left(3^1, 5^{1+u_{m_1-3}}, 6^{v_{m_1-3}}, 7^{w_{m_1-3}}, m_1^{f_{m_1}-1}, \ldots \right). $$
 
Our explanation for why $\phi$ is injective in this case is similar to that in
Case 1.  Suppose that we are given
an element in the range of $\phi$ of the form $\phi(\pi) = \left(3^1, 5^{1+A}, 6^{B}, 7^{C},
m_1^{f_{m_1}-1}, \ldots \right)$.  Then, using \eqref{eq:defatwo}, we can find $m_1 - 3$ 
from $A, B$ and $C$, and from there $m_1$ can be recovered.  From
$m_1$, we can reconstruct the partition $\pi$ uniquely.  

There are cases below where the reasoning that $\phi$ is injective is similar to
this case and Case 1, so we omit the details there.

  Case 2(c)(ii)($\beta$)(I)(B)(ii): $m_1 = 7$.
  
  Case 2(c)(ii)($\beta$)(I)(B)(ii)(a) \etwo*:  $f_7 \geq 2$. Then define $$ \phi(\pi) = \left(3^5, 4^1, 5^0, 7^{f_7-2}, \ldots \right). $$
  
  Case 2(c)(ii)($\beta$)(I)(B)(ii)(b): $f_7 =1$. Then $\pi = \left(5^1, 7^1, 8^{f_8}, \ldots \right)$. Let $m_2 \geq 8$  be the least number with a nonzero frequency in $\pi$. 
  
  Case 2(c)(ii)($\beta$)(I)(B)(ii)(b)(i) \btwo*: $m_2=8$.  Define $$ \phi(\pi) = \left(3^2, 4^1, 5^2, 7^0, 8^{f_8-1}, \ldots \right). $$
  
  Case 2(c)(ii)($\beta$)(I)(B)(ii)(b)(ii) \bthree*: $9 \leq m_2 < L+3$.  Define $$\phi(\pi) = \left(3^2, 4^1, 5^1, (m_2-3)^1, m_2^{f_{m_2}-1}, \ldots \right).$$
  To be clear, our notation indicates that the frequency of $7$ in $\phi(\pi)$
  is 0 unless $m_2 = 10$.

 Case 2(c)(ii)($\beta$)(I)(B)(ii)(b)(iii) \bfive*:   $m_2 = L+3$.  Then  $\pi = \left(5^1,
7^1,  \ldots (L+3)^{\freq_{L+3}} \right)$.  Define  $$\phi(\pi) = \left(3^2, 4^2,
5^2, 7^0, (L-9)^1, (L+3)^{\freq_{L+3}-1} \right).$$

   Case 2(c)(ii)($\beta$)(I)(B)(iii): $m_1 = 11$.  Then $ \pi = \left(5^1, 11^{f_{11}}, \ldots \right)$. We have further subcases.
   
   Case 2(c)(ii)($\beta$)(I)(B)(iii)(a) \ione*: $f_{11} \geq 2$.  Define $$\phi(\pi) = \left(3^9, 5^0, 11^{f_{11}-2}, \ldots \right). $$
   
   Case 2(c)(ii)($\beta$)(I)(B)(iii)(b) \hone*: $f_{11} = 1$.  Then $ \pi = \left(5^1, 11^1, 12^{f_{12}}, \ldots \right)$. Let $m_3 \geq 12$  be the least number with a nonzero frequency in $\pi$. 
   Then define $$ \phi(\pi) = \left(3^8, (m_3-8)^1, m_3^{f_{m_3}-1}, \ldots \right). $$
   
   Case 2(c)(ii)($\beta$)(I)(B)(iv): $m_1 = 12$.  Then $\pi = \left(5^1,
   12^{f_{12}}, \ldots \right)$.   We have further subcases.
   
   Case 2(c)(ii)($\beta$)(I)(B)(iv)(a) \gone*: $f_{12} \geq 2$.  Define $$\phi(\pi) = \left(3^7, 4^2, 5^0, 12^{f_{12}-2}, \ldots \right). $$
   
    Case 2(c)(ii)($\beta$)(I)(B)(iv)(b): $f_{12} = 1$.  Then $ \pi = \left(5^1,
	12^1, 13^{f_{13}}, \ldots \right)$. Let $m_4 \geq 13$  be the least number
	with a nonzero frequency in $\pi$,  so $\pi = \left(5^1, 12^1, m_4^{f_{m_4}}, \ldots \right)$. 
    
    Case 2(c)(ii)($\beta$)(I)(B)(iv)(b)(i) \done*: $m_4 = 13$.  Define $$\phi(\pi) = \left(3^4, 6^3, 13^{f_{13}-1}, \ldots \right). $$
    
Case 2(c)(ii)($\beta$)(I)(B)(iv)(b)(ii) \itwo*: $m_4 \geq 14$.  Then $m_4-10 \geq 4$, and by Lemma \ref{4,5,6,7} there exist nonnegative integers $X_{m_4-10}, Y_{m_4-10}, Z_{m_4-10}$ and $U_{m_4-10}$ such that $$m_4-10 = 4X_{m_4-10}+5Y_{m_4-10}+6Z_{m_4-10}+7U_{m_4-10}. $$
     For each $m_4 \geq 14$, fix a solution to the above equation and define $$ \phi(\pi) = \left(3^9, 4^{X_{m_4-10}}, 5^{Y_{m_4-10}}, 6^{Z_{m_4-10}}, 7^{U_{m_4-10}}, m_4^{f_{m_4}-1}, \ldots, \right). $$
     
     Case 2(c)(ii)($\beta$)(I)(C): $f_5 =2$. Thus $\pi = \left(5^2,7^{f_7}, \ldots, \right)$. Let $m_5 \geq 7$  be the least number with a nonzero frequency in $\pi$. 
     
	 Case 2(c)(ii)($\beta$)(I)(C)(i) \athree*: $m_5 \neq 10$. Then $m_5 - 3 \geq
	 4$ and $m_5-3 \neq 7$.  By Lemma \ref{4,5,6} there are nonnegative integers $x_{m_5-3}, y_{m_5-3}$ and $z_{m_5-3}$ of the equation $$ m_5-3 = 4x_{m_5-3} + 5y_{m_5-3} + 6z_{m_5-3}. $$  For each $m_5 \geq 7$ such that $m_5 \neq 10$, fix a solution to the above equation and define $$ \phi(\pi) = \left(3^1, 4^{1+{x_{m_5-3}}}, 5^{y_{m_5-3}}, 6^{1+z_{m_5-3}}, m_5^{f_{m_5-1}}, \ldots \right). $$
     
      Case 2(c)(ii)($\beta$)(I)(C)(ii): $m_5 = 10$.  Then $\pi = (5^2, 10^{f_{10}}, \ldots) $. 

      Case 2(c)(ii)($\beta$)(I)(C)(ii)(a) \jone*: $f_{10} \geq 2$. Then define $$ \phi(\pi) = \left(3^{10}, 10^{f_{10}-2}, \ldots \right). $$
      
      Case 2(c)(ii)($\beta$)(I)(C)(ii)(b): $f_{10} =1$. Then $\pi = (5^2, 10^1, 11^{f_{11}}, \ldots)$. Let $m_6 \geq 11$  be the least number with a nonzero frequency in $\pi$. 
      
      Case 2(c)(ii)($\beta$)(I)(C)(ii)(b)(i) \gtwo*: $m_6$ is odd. Then define $$ \phi(\pi) = \left(3^7, \left(\frac{m_6-1}{2} \right)^2, m_6^{f_{m_6}-1}, \ldots \right). $$
      
       Case 2(c)(ii)($\beta$)(I)(C)(ii)(b)(ii) \gthree*: $m_6$ is even. Then define $$ \phi(\pi) = \left(3^7, \left(\frac{m_6}{2}-1 \right)^1, \left(\frac{m_6}{2} \right)^1,  m_6^{f_{m_6}-1}, \ldots \right). $$
       
       Case 2(c)(ii)($\beta$)(II): $f_5=f_6=0$.  Since $s(\pi) \leq 6$, we have $f_4 \geq 1$. Thus $\pi = \left(4^{f_4}, 7^{f_7}, \ldots \right)$.
  
  Case 2(c)(ii)($\beta$)(II)(A) \ftwo*: $f_4 \geq 3$. Define $$\phi(\pi) = \left(3^4, 4^{{f_4}-3}, 7^{f_7}, \ldots \right).$$
 
  Case 2(c)(ii)($\beta$)(II)(B): $f_4 = 1$.  Thus $\pi = \left(4^1, 7^{f_7}, \ldots \right)$. Let $m_7 \geq 7$ be the least number with a nonzero frequency in $\pi$.  Thus $\pi = \left(4^1, m_7^{f_{m_7}}, \ldots \right)$. 
  
  Case 2(c)(ii)($\beta$)(II)(B)(i) \afour*: $m_7 \neq 10, 14$. Then $m_7 - 3
  \geq 4$ and $m_7-3 \neq 7, 11$.  By Lemma \ref{4,5,6} there is a triple $(x_{m_7-3}, y_{m_7-3}, z_{m_7-3})$ such that $$ m_7-3 = 4x_{m_7-3}+5y_{m_7-3}+6z_{m_7-3}.$$ 
  
  Crucially, to avoid injectivity problems with the Case 2(c)(ii)($\beta$)(I)(C)(i), if $m_7 = m_5+6$, using Lemma \ref{twosol}, we choose a solution such that $(x_{m_7-3}, y_{m_7-3}, z_{m_7-3}) \neq (x_{m_5-3}, y_{m_5-3}, 1+z_{m_5-3})$.
  Note here that $m_7 \neq 14$, so $m_5 - 3 \neq 5$, which is required to use Lemma \ref{twosol}.
  
  Define $$\phi(\pi) = \left(3^1, 4^{1+x_{m_7-3}}, 5^{y_{m_7-3}}, 6^{z_{m_7-3}}, \ldots, m_7^{f_{m_7}-1}, \ldots \right).$$
  
  Case 2(c)(ii)($\beta$)(II)(B)(ii): $m_7 = 10$. Then $\pi = \left(4^1, 10^{f_{10}}, \ldots \right)$. 
  
   Case 2(c)(ii)($\beta$)(II)(B)(ii)(a) \dthree*: $f_{10} \geq 2$. Define $$\phi(\pi) = \left(3^4, 6^2, 10^{f_{10}-2}, \ldots \right). $$
  
  Case 2(c)(ii)($\beta$)(II)(B)(ii)(b) \gfour*: $f_{10}=1$. Thus $\pi = \left(4^1,10^1,11^{f_{11}}, \ldots \right)$. Let $m_8 \geq 11$ be the least number with a nonzero frequency in $\pi$. 
  Then define $$\phi(\pi) = \left(3^7, (m_8-7)^1, m_8^{f_{m_8}-1}, \ldots \right).$$
  
  Case 2(c)(ii)($\beta$)(II)(B)(iii) \ftwo*: $m_7 = 14$. So $\pi = \left(4^1, 14^{f_{14}}, \ldots \right)$. Define $$ \phi(\pi) = \left(3^6, 4^0,14^{f_{14}-1}, \ldots \right). $$
  
  Case 2(c)(ii)($\beta$)(II)(C): $f_4 = 2$.  Thus $\pi = \left(4^2, 7^{f_7}, \ldots \right)$. Let $m_9 \geq 7$ be the least number with a nonzero frequency in $\pi$. 
  
  Case 2(c)(ii)($\beta$)(II)(C)(i) \afive*: $m_9$ is odd.  Define $$ \phi(\pi) = \left(3^1, \left(\frac{m_9+5}{2} \right)^2, m_9^{f_{m_9-1}}, \ldots \right). $$
  
  Case 2(c)(ii)($\beta$)(II)(C)(ii) \asix*: $m_9$ is even. Define $$ \phi(\pi) = \left(3^1, \left(\frac{m_9}{2} +2\right)^1, \left(\frac{m_9}{2} + 3 \right)^1, m_9^{f_{m_9-1}}, \ldots  \right). $$

To prove the injectivity of the map $\phi$, we organize the cases based on the various frequencies of $3$ in $\phi(\pi)$.

First we organize the cases where the frequency of $3$ in $\phi(\pi)$ is
$1$.\label{page:pageas}

\begin{enumerate}
\item[\aone.] Case 2(b): $\phi(\pi) = \left(3^1, (s(\pi)-3)^1, (s(\pi)^{f_{s(\pi)}-1}), \ldots, \right)$, where $s(\pi) \geq 7$.
\item[\atwo.] Case 2(c)(ii)($\beta$)(I)(B)(i): $\phi(\pi) = \left(3^1, 5^{1+A}, 6^{B}, 7^{C}, m_1^{f_{m_1}-1}, \ldots \right)$, where $m_1 \geq 8$, $m_1 \neq 11,12$, and $A,B$ and $C$ are some nonnegative integers such that at least one of these is positive.
\item[\athree.] Case 2(c)(ii)($\beta$)(I)(C)(i): $ \phi(\pi) = \left(3^1, 4^{1+{x_{m_5-3}}}, 5^{y_{m_5-3}}, 6^{1+z_{m_5-3}}, m_5^{f_{m_5-1}}, \ldots \right)$, where $m_5 \geq 7, m_5 \neq 10$, and $4x_{m_5-3}+5y_{m_5-3}+6z_{m_5-3} = m_5-3$.
\item[\afour.] Case 2(c)(ii)($\beta$)(II)(B)(i): $\phi(\pi) = \left(3^1, 4^{1+x_{m_7-3}}, 5^{y_{m_7-3}}, 6^{z_{m_7-3}}, \ldots, m_7^{f_{m_7}-1}, \ldots \right)$, where $m_7 \geq 7, m_7 \neq 10,14$, and $4x_{m_7-3}+5y_{m_7-3}+6z_{m_7-3} = m_7-3$. Moreover, 
 if $m_7 = m_5+6$, then $(x_{m_7-3}, y_{m_7-3}, z_{m_7-3}) \neq (x_{m_5-3}, y_{m_5-3}, 1+z_{m_5-3})$.
\item[\afive.] Case 2(c)(ii)($\beta$)(II)(C)(i):  $ \phi(\pi) = \left(3^1, \left(\frac{m_9+5}{2} \right)^2, m_9^{f_{m_9-1}}, \ldots \right)$, where $m_9 \geq 7$ is odd.
\item[\asix.] Case 2(c)(ii)($\beta$)(II)(C)(ii):  $ \phi(\pi) = \left(3^1, \left(\frac{m_9}{2} +2\right)^1, \left(\frac{m_9}{2} + 3 \right)^1, m_9^{f_{m_9-1}}, \ldots  \right)$, where $m_9 \geq 7$ is even.
\end{enumerate}

Each case above is individually injective (we can find $\pi$ from $\phi(\pi)$).
We explain why the map $\phi$ is injective overall so far, and we do this by confirming that no two distinct
cases contain common partitions.  In Case \aone, the second
smallest and the third smallest parts differ by at least $3$, and the frequency
of the second smallest part is $1$. This distinguishes it from all the other cases.
In Case \atwo, the number $4$ is not present as a part, which distinguishes it from
Cases \athree and \afour, and it contains $5$ as a part, which distinguishes it
from Cases \afive and \asix. In Cases \athree and \afour, the number $4$ is
present as a part, which distinguishes it from Cases \afive and \asix.  Cases
\afive and \asix are distinguished by the frequency of the second smallest part. 

What remains is to show that Cases \athree and \afour can be distinguished, and
we show this by demonstrating the cases have no common element in
the image of $\phi$.  Suppose, to the contrary, that Cases \athree and \afour have a common element.  Then
$(x_{m_7-3}, y_{m_7-3}, z_{m_7-3}) = (x_{m_5-3}, y_{m_5-3}, 1+z_{m_5-3})$. From
the relations $4x_{m_5-3}+5y_{m_5-3}+6z_{m_5-3} = m_5-3$ and
$4x_{m_7-3}+5y_{m_7-3}+6z_{m_7-3} = m_7-3$, we obtain $m_7 = m_5+6$. But if $m_7 = m_5+6$, then $(x_{m_7-3}, y_{m_7-3}, z_{m_7-3}) \neq
(x_{m_5-3}, y_{m_5-3}, 1+z_{m_5-3})$, giving the required contradiction.

We follow the same reasoning below.  We collect cases according to the frequency of
3 in partitions in the range of $\phi$.  We then explain why
all the distinct cases with fixed frequency of 3 have no partitions in common in their
range.  We leave the verification that each of the different cases for fixed
frequency of 3 in the range of $\phi$ are individually injective to the reader.

Next we organize the cases where the frequency of $3$ in $\phi(\pi)$ is $2$.

\begin{itemize}
	\item[\bnewone.] Case 2(a) :  $\phi(\pi) = \left(3^2, 4^2, 5^1, (L-16), \ldots, (L+3)^{\freq_{L+3} - 1}\right).$

	\item[\bone.] Case 2(c)(ii)($\alpha$): $\phi(\pi)=(3^2,4^{f_4},5^{f_5}, \ldots)$, where $f_4=0$ or $f_5 = 0$.
	\item[\btwo.]  Case 2(c)(ii)($\beta$)(I)(B)(ii)(b)(i): $ \phi(\pi) =
		\left(3^2, 4^1, 5^2, 8^{f_8-1},\ldots \right)$, where $f_8 \geq 1$.
	\item[\bthree.] Case 2(c)(ii)($\beta$)(I)(B)(ii)(b)(ii): $ \phi(\pi) =
		\left(3^2, 4^1, 5^1, (m_2-3)^1, m_2^{f_{m_2}-1}, \ldots \right)$, where
		$9 \geq m_2 < L+3$.
	\item[\bfive.] Case 2(c)(ii)($\beta$)(I)(B)(ii)(b)(iii):  $\phi(\pi) = \left(3^2, 4^2,
5^2, 7^0, (L-9)^1, (L+3)^{\freq_{L+3}-1} \right).$
\end{itemize}
These cases are distinguished by their frequencies of $4$ and $5$.

There is only one case where the frequency of $3$ in $\phi(\pi)$ is $3$:
\begin{itemize}
	\item[\cone.] Case 2(c)(i): $\phi(\pi) = \left(3^3, 4^{f_4-1},5^{f_5-1}, 6^{f_6}, \ldots \right)$.
\end{itemize}
So it is distinguishable from other cases.

Next we organize the cases in which the frequency of $3$ in $\phi(\pi)$ is $4$.

\begin{itemize}
	\item[\done.] Case 2(c)(ii)($\beta$)(I)(B)(iv)(b)(i):  $ \phi(\pi) =
		\left(3^4, 6^3, 13^{f_{13}-1}, \ldots \right)$, where $f_{13} \geq 1$.
	\item[\dtwo.] Case 2(c)(ii)($\beta$)(II)(A): $\phi(\pi) = \left(3^4, 4^{{f_4}-3}, 7^{f_7}, \ldots \right) $, where $f_4 \geq 3$.
	\item[\dthree.] Case 2(c)(ii)($\beta$)(II)(B)(ii)(a): $\phi(\pi) =
		\left(3^4, 6^2, 10^{f_{10}-2}, \ldots \right)$, where $f_{10} \geq 2$.
\end{itemize}

Thus, when the frequency of $3$ in $\phi(\pi)$ is $4$, these cases are
distinguishable by the frequency of 6 in the image.
Next we organize the cases in which the frequency of $3$ in $\phi(\pi)$ is $5$.

\begin{itemize}
	\item[\eone.] Case 2(c)(ii)($\beta$)(I)(A):  $ \phi(\pi)= \left(3^5,
			5^{f_5-3}, 7^{f_7}, \ldots\right)$, where $f_5 \geq 3$.
	\item[\etwo.] Case 2(c)(ii)($\beta$)(I)(B)(ii)(a):  $ \phi(\pi) = \left(3^5, 4^1, 7^{f_7-2}, \ldots \right)$, where $f_7 \geq 2$ .
\end{itemize}

Thus, when the frequency of $3$ in $\phi(\pi)$ is $5$, these cases are
distinguishable by the frequency of 4 in the image.
Next we organize the cases in which the frequency of $3$ in $\phi(\pi)$ is $6$.

\begin{itemize}
	\item[\fone.] Case 1 with $f=1$: $\phi (\pi) = \left(3^6, 4^{\alpha},
		5^{\beta}, 6^{\gamma}, 7^{\delta}, \ldots  \right)$, where $\alpha,
		\beta, \gamma$ and $\delta$ are nonnegative integers with at least one positive.
	\item[\ftwo.] Case 2(c)(ii)($\beta$)(II)(B)(iii): $ \phi(\pi) = \left(3^6, 14^{f_{14}-1}, \ldots \right)$, where $f_{14} \geq 1$.
\end{itemize}

Thus, when the frequency of $3$ in $\phi(\pi)$ is $6$, these cases are distinguishable.
Next we organize the cases in which the frequency of $3$ in $\phi(\pi)$ is $7$.

\begin{itemize}
	\item[\gone.] Case 2(c)(ii)($\beta$)(I)(B)(iv)(a):  $ \phi(\pi) = \left(3^7, 4^2, 12^{f_{12}-2}, \ldots \right)$, where $f_{12} \geq 2$.
	\item[\gtwo.] Case 2(c)(ii)($\beta$)(I)(C)(ii)(b)(i):  $ \phi(\pi) = \left(3^7, \left(\frac{m_6-1}{2} \right)^2, m_6^{f_{m_6}-1}, \ldots \right)$, where $m_6 \geq 11$ is odd.
	\item[\gthree.] Case 2(c)(ii)($\beta$)(I)(C)(ii)(b)(ii): $ \phi(\pi) = \left(3^7, \left(\frac{m_6}{2}-1 \right)^1, \left(\frac{m_6}{2} \right)^1,  m_6^{f_{m_6}-1}, \ldots \right)$, where $m_6 \geq 11$ is even.
	\item[\gfour.] Case 2(c)(ii)($\beta$)(II)(B)(ii)(b): $\phi(\pi) = (3^7, (m_8-7)^1, m_8^{f_{m_8}-1}, \ldots)$, where $m_8 \geq 11$.
\end{itemize}

Thus, when the frequency of $3$ in $\phi(\pi)$ is $7$, the next parts after $3$ and their frequencies distinguish the various cases.

There is only one case in which the frequency of $3$ in $\phi(\pi)$ is $8$: 

\begin{itemize}
	\item[\hone.] Case 2(c)(ii)($\beta$)(I)(B)(iii)(b): $ \phi(\pi) =
		\left(3^8, (m_3-8)^1, m_3^{f_{m_3}-1}, \ldots \right)$, where $m_3 \geq
		12$ and $\freq_{m_3} \geq 1$.
\end{itemize}
Next we organize the cases in which the frequency of $3$ in $\phi(\pi)$ is $9$.

\begin{itemize}
	\item[\ione.]   Case 2(c)(ii)($\beta$)(I)(B)(iii)(a): $\phi(\pi) = \left(3^9, 11^{f_{11}-2}, \ldots \right)$, where $f_{11} \geq 2$
	\item[\itwo.]   Case 2(c)(ii)($\beta$)(I)(B)(iv)(b)(ii): $ \phi(\pi) = \left(3^9, 4^{\alpha}, 5^{\beta}, 6^{\gamma}, 7^{\delta}, m_4^{f_{m_4}-1}, \ldots, \right)$, where $\alpha, \beta, \gamma$ and $\delta$ are nonnegative integers such that at least one of these is positive and $m_4 \geq 14$.
\end{itemize}

Thus, when the frequency of $3$ in $\phi(\pi)$ is $9$, these cases are distinguishable.

There is only case in which the frequency of $3$ in $\phi(\pi)$ is $10$. 
\begin{itemize}
	\item[\jone.] Case 2(c)(ii)($\beta$)(I)(C)(ii)(a): $ \phi(\pi) = \left(3^{10}, 10^{f_{10}-2}, \ldots \right)$, where $f_{10} \geq 2$.
\end{itemize}

Finally, there is only one case in which the frequency of $3$ in $\phi(\pi)$ is
$6f$ for some $f \geq 2$.
\begin{itemize}
	\item[\kone.] Case $1$ with $f \geq 2$. $\phi (\pi) = \left(3^6, 4^{\alpha}, 5^{\beta}, 6^{\gamma}, 7^{\delta}, \ldots  \right)$, where $\alpha, \beta, \gamma$ and $\delta$ are some nonnegative integers such that at least one of these is positive.
\end{itemize}

Hence all the cases are distinguishable, and the map $\phi$ is injective.
This shows nonnegativity of the coefficient of $q^N$  in $H_{L,3,L}(q)$
when $L \geq 22$ and $N \geq 21$.  To show these coefficients are
positive, we find an element of the codomain of $\phi$ that is not in
its range.  In all cases such an element will have frequency of 3
equalling 4, and these can be compared to  \done* -
\dthree* to ensure they are not in the range of $\phi$.  If $N=21$ and $N=22$, the partitions $\pi_{21} = (3^4, 4^1, 5^1)$ and
$\pi_{22} = (3^4, 5^2)$ are partitions not in the range but in the codomain of $\phi$.

For the remaining cases, we need the following result:  for any positive integer
$n \geq 6$, the equation
\begin{equation*}
	6x_6 + 7x_7 + \cdots + 11x_{11} = n
\end{equation*}
has a solution where $x_6, \ldots, x_{11}$ are nonnegative integers (see
\cite[Lemma 8]{BR20} for a proof of a more general result).  If $N \geq
23$, then $N-17 \geq 6$, so we can fix a solution with nonnegative integers
to the equation
\begin{equation*}
	6x_6 + 7x_7 + \cdots + 11x_{11} = N - 17.
\end{equation*}
Then the partition $\pi_N = (3^4, 5^1, 6^{x_6}, \ldots, 11^{x_{11}}, \ldots)$ is
not in the range but in the codomain of $\phi$.

\end{proof}

We are left with the cases $4 \leq L \leq 21$. We deal with $7 \leq L \leq 21$
in Lemma \ref{Two2} below. The proof of this lemma is similar in spirit to the
proof of \cite[Theorem 19]{BR20}.

\begin{lemma}
\label{Two2}
Let $7 \leq L \leq 21$ and $N_L= L^2+10L+7$.  Then the coefficient of $q^N$ in
$H_{L,3,L}(q)$ is nonnegative whenever $N \geq N_L $.
\end{lemma}

\begin{proof}
	Again it suffices to show that for fixed $7 \leq L \leq 21$ and $N \geq N_L
	$, there
	is an injective map $\phi$ as in \eqref{eq:gamma}.   Let $\pi =
\left(4^{f_4}, \ldots, L^{f_L}, \ldots, (L+3)^{f_{L+3}}\right)$ be a partition
of $N$ in $D_{L,3}$ and let $f = f_L$.  We define $\phi(\pi)$ in cases depending on the
value of $f$.  

Case $1$: $f$ is positive and $f \equiv 0$ (mod $3$).  Then define $$ \phi(\pi) =
\left(3^{\frac{L f}{3}}, 4^{f_4}, \ldots, L^0, \ldots, (L+3)^{f_{L+3}}\right).$$

Case 2: $f \equiv 1$ (mod $3$). Then define $$ \phi(\pi) =
\left(3^{L\left(\frac{f-1}{3}\right)+1}, 4^{f_4}, \ldots, (L-3)^{f_{L-3}+1}, \ldots, L^0,
\ldots, (L+3)^{f_{L+3}}\right).$$

Case $3$: $f \equiv 2$ (mod $3$). Then define $$ \phi(\pi) =
\left(3^{L\left(\frac{f-2}{3}\right)+2}, 4^{f_4}, \ldots, (L-3)^{f_{L-3}+2}, \ldots, L^0,
\ldots, (L+3)^{f_{L+3}}\right).$$

Case $4$: $f = 0$. Since $N \geq N_L$ is large enough, either $f_{L+2} \geq 6$ or there exists an $i \neq L+2$ such
that $4 \leq i \leq L+3$ and $f_i \geq 3$. Note that the condition on $N$ is in fact tight for this to happen. We have further subcases.

Case $4$(i): $f_{L+2} \geq 6$. Then define $$ \phi(\pi) = \left(3^{2L+4}, 4^{f_4}, \ldots L^0, (L+2)^{f_{L+2}-6}, (L+3)^{f_{L+3}}\right). $$

Case $4$(ii): $f_{L+2} \leq 5$, and there exists an $i \neq L+2$ such
that $4 \leq i \leq L+3$ and $f_i \geq 3$. Let $i_0$ be the least such number. Note 
$i_0 \neq L$ since $f = 0$. We have further subcases depending on whether $i_0 = L+1$
or not. 

Case $4$(ii)(a): $i_0 \neq L+1$. Then define $$ \phi(\pi) = \left(3^{i_0},
4^{f_4}, \ldots, i_0^{f_{i_0}-3}, \ldots L^0, \ldots, (L+3)^{f_{L+3}}\right).$$

Case $4$(ii)(b): $i_0 = L+1$. Then define $$  \phi(\pi) = \left(3^3, 4^{f_4},
\ldots, (L-2)^{f_{L-2}+3}, L^0, (L+1)^{f_{L+1}-3}, \ldots, (L+3)^{f_{L+3}}\right).$$

It is easy to see that $\phi$ is injective in each case.  To see that $\phi$ is injective overall, 
note that the frequency of $3$ modulo $L$ in the image distinguishes the cases,
with the exception of Cases 4(i) and 4(ii)(a) (when $i_0 = 4$):  in these two cases the frequency
of 3 is 4 modulo $L$, but in Case 4(ii)(a) the frequency is precisely 4, whereas
in Case 4(i) the frequency is at least $L + 4$.  Thus all cases are
distinguishable.  Hence the map $\phi$ is injective. 

\end{proof}
  
 We handle the cases $4 \leq L \leq 6$ in the next three lemmas.
 
 \begin{lemma}
 \label{Three}
The coefficient of $q^N$ in $H_{6,3,6}(q)$ is nonnegative whenever $N \geq 67$.
 \end{lemma} 
 
 \begin{proof}
 
	 Again it suffices to show that for $L = 6$ and fixed $N \geq 67$, there is an injective
	 map $\phi$ as in \eqref{eq:gamma}.  Recall that partitions of $N$ in $I_{6,3,6}$, the codomain
	 of $\phi$,
have smallest part $3$, no part equal to $6$, and 
largest part at most 9. Let $\pi = \left(4^{f_4}, \ldots, 6^{f_6},
\ldots, 9^{f_9}\right)$ be a partitions of $N$ in $D_{6,3}$ and let $f=f_6$.  We define
$\phi(\pi)$ in cases depending on the value of $f$.

Case $1$: $f > 0$. Define $$ \phi(\pi) = \left(3^{2f}, 4^{f_4}, 5^{f_5}, 6^0, 7^{f_7}, 8^{f_8}, 9^{f_9} \right). $$

Case $2$: $f=0$. Then $\pi = (4^{f_4},5^{f_5}, 6^0, 7^{f_7}, 8^{f_8}, 9^{f_9})$. Since $N \geq 67$, there exists $ 4 \leq i \leq 9$ such that $i \neq 6$ and
$f_i \geq 3$. Let $i_0$ be the least such number. Note that the condition on $N$ is tight for this to happen.

Case $2$(i): $i_0$ is odd. So $i_0$ is $5,7$ or $9$. Define $$ \phi(\pi) = \left(3^{i_0}, \ldots, 6^0, \ldots  i_0^{f_{i_0}-3}  \ldots \right). $$

Case $2$(ii): $i_0 = 4$.  Define $$ \phi(\pi) = \left(3^1, 4^{f_4-3}, 5^{f_5}, 6^0, 7^{f_7}, 8^{f_8}, 9^{f_9+1} \right). $$

Case $2$(iii): $i_0 = 8$.  Define $$ \phi(\pi) = \left(3^3, 4^{f_4}, 5^{f_5+3}, 6^0, 7^{f_7}, 8^{f_8-3}, 9^{f_9} \right). $$

It is easy to see that $\phi$ is injective in each case.  To see that $\phi$ is injective overall, 
note that the frequency of $3$ in the image distinguishes the cases. Hence the map $\phi$ is injective.

 \end{proof}
 
 \begin{lemma}
 \label{Four}
The coefficient of $q^N$ in $H_{5,3,5}(q)$ is nonnegative whenever $N \geq 164$.
 \end{lemma}
 
 \begin{proof}
 
	 Again it suffices to show that for $L = 5$ and fixed $N \geq 164$, there is an injective map
$\phi$ as in \eqref{eq:gamma}.  Recall that partitions of $N$ in $I_{5,3,5}(q)$,
the codomain of $\phi$, have smallest part 3, no part equal to 5, and largest part
 at most 8.  Let $\pi = \left(4^{f_4}, 5^{f_5}, \ldots,
 8^{f_8}\right)$ be a partition of $N$ in $D_{5,3}$ and let $f$ denote $f_5$.  We define
$\phi(\pi)$ in cases depending on the value of $f$.

 Case $1$: $f$ is a positive number with $f \equiv 0$ (mod $3$). Define $$ \phi(\pi) = \left(3^{\frac{5f}{3}}, 4^{f_4}, 5^0, 6^{f_6}, 7^{f_7}, 8^{f_8} \right). $$
 
Case $2$:  $f>1$ and $f \equiv 1$ (mod $3$). Define $$ \phi(\pi) = \left(3^{5\left(\frac{f-4}{3}\right)+4}, 4^{f_4+2}, 5^0, 6^{f_6}, 7^{f_7}, 8^{f_8} \right). $$

Case $3$:  $f \equiv 2$ (mod $3$). Define $$ \phi(\pi) =
\left(3^{5\left(\frac{f-2}{3}\right)+1}, 4^{f_4}, 5^0, 6^{f_6}, 7^{f_7+1},
8^{f_8} \right). $$

We are left with the cases $f=0$ and $f=1$.

Case $4$: Suppose $f=0$. Then $\pi = \left(4^{f_4}, 5^0, 6^{f_6}, 7^{f_7}, 8^{f_8} \right)$. Since $N \geq 164$ is large enough, at least one of the following 
conditions is true: (i) $f_4 \geq 6$;  (ii) $f_6 \geq 1$;  (iii) $f_7 \geq 3$;
or (iv) $f_8 \geq 12$.  We deal with each case below.  Note that the condition on $N$ is not tight. The bound of $164$ will be required in Case $5$.

Case $4$(i): $f_4 \geq 6$. Define $$ \phi(\pi) = \left(3^8, 4^{f_4-6}, 5^0, 6^{f_6}, 7^{f_7},8^{f_8} \right). $$

Case $4$(ii): $f_4 \leq 5$ and $f_6 \geq 1$. Define $$ \phi(\pi) = \left(3^2, 4^{f_4}, 5^0, 6^{f_6-1}, 7^{f_7},8^{f_8} \right). $$

Case $4$(iii): $f_4 \leq 5$, $f_6 = 0$ and $f_7 \geq 3$. Define $$ \phi(\pi) = \left(3^7, 4^{f_4}, 5^0, 7^{f_7-3},8^{f_8} \right). $$

Case $4$(iv): $f_4 \leq 5$, $f_6 = 0$, $f_7 \leq 2$ and $f_8 \geq 12$. Define $$ \phi(\pi) = \left(3^{32}, 4^{f_4}, 5^0, 7^{f_7},8^{f_8-12} \right). $$

Case $5$: $f=1$. Then $\pi = \left(4^{f_4}, 5^1, 6^{f_6}, 7^{f_7}, 8^{f_8}
\right)$. Since $N \geq 164$ is large enough, at least one of the following 
conditions is true: (i) $f_4 \geq 1$;  (ii) $f_6 \geq 11$; (iii) $f_7 \geq 7$;
or (iv) $f_8 \geq 8$.  We deal with each case below.  Note that the condition on $N$ is tight here.  
Case $5$(i): $f_4 \geq 1$. Define $$ \phi(\pi) = \left(3^3, 4^{f_4-1}, 5^0, 6^{f_6}, 7^{f_7},8^{f_8} \right). $$

Case $5$(ii): $f_4 = 0$ and $f_6 \geq 11$. Define $$ \phi(\pi) = \left(3^{13}, 4^8, 5^0, 6^{f_6-11}, 7^{f_7},8^{f_8} \right). $$

Case $5$(iii): $f_4 = 0$, $f_6 \leq 10$ and $f_7 \geq 7$. Define $$ \phi(\pi) = \left(3^{18}, 4^0, 5^0, 6^{f_6}, 7^{f_7-7},8^{f_8} \right). $$

Case $5$(iv): $f_4 = 0$, $f_6 \leq 10$, $f_7 \leq 6$ and $f_8 \geq 8$ . Define $$ \phi(\pi) = \left(3^{23}, 4^0, 5^0, 6^{f_6}, 7^{f_7},8^{f_8-8} \right). $$

It is easy to see that $\phi$ is injective in each case.  To see that $\phi$ is injective overall, 
note that the frequency of $3$ in the image distinguishes the cases. In Cases $1$, $2$ and $3$, the frequency of $3$ is $0, 4$ and $1$ modulo $5$, respectively. In Cases $4$ and $5$, it is always $2$ or $3$ modulo $5$ and different for each subcase. Hence the map $\phi$ is injective.

\end{proof}

\begin{lemma}
\label{Five}
The coefficient of $q^N$ in $H_{4,3,4}(q)$ is nonnegative whenever $N \geq 1042$.
\end{lemma}

\begin{proof}

	Again it suffices to show that for $L = 4$ and fixed $N \geq 1042$, there is an
injective map $\phi$ as in \eqref{eq:gamma}.  Recall that partitions of $N$ in
$I_{4,3,4}(q)$, the codomain of $\phi$, have smallest part $3$, no part equal to $4$, and largest part at most 7.  Let $\pi =
(4^{f_4},5^{f_5}, 6^{f_6}, 7^{f_7})$ be a partitions of $N$ in $D_{4,3}$ and let $f =
f_4$.  We define $\phi(\pi)$ in cases depending on $f$.

For $n \geq 10$, Lemma \ref{5,6,7} guarantees that there exists nonnegative integer solutions $(x_n,y_n,z_n)$ of the equation 
\begin{equation}\label{eq:definvert}
	n = 5x_n+6y_n+7z_n.
\end{equation}
For each $n \geq 10$, fix a nonnegative integer solution $(x_n,y_n,z_n)$ to the equation. 

Case $1$: $10 \leq f < 100$. Define $$ \phi(\pi) = (3^f, 4^0, 5^{f_5+x_f}, 6^{f_6+y_f}, 7^{f_7+z_f}). $$
It is easy to see that $\phi$ is injective in this case:  given an element
$\phi(\pi)$  whose frequency is between 10 and 100, the frequency of 3
gives $f$, and the values of $x_f, y_f$ and $z_f$ can be found from
\eqref{eq:definvert}.  The partition $\pi$ can then be reconstructed.  Similar
arguments can be used to show that $\phi$ is injective in the other cases below.

Case $2$: $f \geq 100$. Define $$ \phi(\pi) = (3^{f+30}, 4^0, 5^{f_5+x_{f-90}},  6^{f_6+y_{f-90}},  7^{f_7+z_{f-90}}). $$

Case $3$: $0 \leq f \leq 9$. Since $N \geq 1042$ is large enough, at least one of the following 
conditions is true: (i) $f_5 \geq 62$;  (ii) $f_6 \geq 57$; or (iii) $f_7 \geq
53$.  We deal with each case below.  Note that the condition on $N$ is in fact tight here.

Case $3$(i): $f_5 \geq 62$. Define $$ \phi(\pi) = \left(3^{f+100}, 4^0, 5^{f_5-62+x_{f+10}}, 6^{f_6+y_{f+10}}, 7^{f_7+z_{f+10}} \right). $$

Case $3$(ii): $f_5 \leq 61$ and $f_6 \geq 57$. Define $$ \phi(\pi) = \left(3^{f+110}, 4^0, 5^{f_5+x_{f+12}}, 6^{f_6-57+y_{f+12}}, 7^{f_7+z_{f+12}} \right). $$

Case $3$(iii): $f_5 \leq 61$, $f_6 \leq 56$ and $f_7 \geq 53$. Define $$ \phi(\pi) = \left(3^{f+120}, 4^0, 5^{f_5+x_{f+11}}, 6^{f_6+y_{f+11}}, 7^{f_7-53+z_{f+11}} \right). $$

It is easy to see that $\phi$ is injective in each case.  To see that $\phi$ is injective overall, 
note that the frequency of $3$ in the image distinguishes the cases. Hence the map $\phi$ is injective. 

\end{proof}

Lemmas \ref{Helpful2}, \ref{Two2}, \ref{Three}, \ref{Four} and \ref{Five} show
that for N larger than a small number, the coefficient of $q^N$ in
$H_{L,s,L}(q)$ is nonnegative.  With these results, the use of computer
searches, and Theorem \ref{Generals} (with $s=3$), we can now prove Theorem \ref{GLthree}.

\begin{proof}[Proof of Theorem \ref{GLthree}]
Let $H_{L,3,L}(q) = \sum_{N \geq 0} a_{L,N} q^N $ and $G_{L,3}(q) = \sum_{N \geq 0} b_{L,N} q^N$.  
By Theorem \ref{Generals},
\begin{equation}\label{eq:blnaln}
	b_{L,N} = a_{L,N} + a_{L,N-L} + a_{L,N-2L} + \cdots.
\end{equation}
	
First we focus on the case $L \geq 22$.  By Lemma \ref{Helpful2}, the
	coefficients satisfy $a_{L,N} \geq 1$ whenever $N \geq 21$.  For $N \leq 20$, we
	observe that $a_{L,N}$ is independent of $L$. To see why, the
	combinatorial interpretation of $a_{L,N}$ in \eqref{eq:comb} gives that,
	when $L \geq N + 1$,  the number $a_{L,N}$  is the difference between the number of partitions of $N$ with
	smallest part $3$ and the number of partitions of $N$ with smallest part at
	least $4$; that is, the condition on the largest part of the partitions
	becomes superfluous.  Thus, when $1 \leq N \leq 20$, since $L \geq N + 1$, we have $a_{L, N} =
	a_{N+1,N}$, and these 20 values can all be found by a computer search.   The search
	finds that $a_{N+1,N}$ is negative only when $N$ is one of $4,5,8,10,12,14$ or $16$, and
	in each case $a_{N+1,N}$ is exactly $-1$.  Thus, for any $N \geq 1$, the right hand side of
	\eqref{eq:blnaln} contains at most one term equal to -1, while the rest of
	the terms are
	positive.  It follows from \eqref{eq:blnaln} that $b_{L,N}$ is negative only when $N$ is one of $4,5,8,10,12,14$ or $16$, and in each case $b_{L,N}$ is exactly $-1$.  This gives Theorem \ref{GLthree} for $L \geq 22$.

The remaining cases are easier.  For $7 \leq L \leq 21$, Lemma \ref{Two2}
renders the unknown values of $a_{L,N}$ to the cases $N \leq N_L$, a finite set
of values that can be searched using a computer.  These computations along with
\eqref{eq:blnaln} give Theorem \ref{GLthree}.  Similarly, for $4 \leq L
\leq 6$, the Lemmas \ref{Three}, \ref{Four} and \ref{Five} leave only a finite
number of unknown values for $a_{L,N}$ for small $N$, which can all be found
using a computer.  These values and an application of \eqref{eq:blnaln} complete the proof of Theorem \ref{GLthree}. 
\end{proof}

\begin{remark}
The programming for $7 \leq L \leq 21$ and $N \leq N_L$ turned out to be a
difficult task in Magma. For example, it is hard to calculate the number of
partitions of $250$ with all parts in the set \{4,5, \ldots, 17\} using Magma.
The command $Partitions(250, min\_part=4, max\_part =17). cardinality()$ in Sage
also does not work (it takes too long and ultimately stops working). We overcame
this problem through another related command and some mathematics. In Sage, we
noticed that the command $Partitions(n, max\_part =17). cardinality()$ is very fast even for large $n$ (even until $n=1000000$, it is fast!) Thus, we calculate the number of partitions with all parts in the set $\{4,5, \ldots, 17\}$  in terms of the number of partitions $p_{17}(n)$ of $n$ with maximum part at most $17$. We do this by viewing partitions with all parts in the set $\{4,5, \ldots, 17\}$ as partitions with maximum part $17$ and no part $1$, $2$ and $3$. Let $A$, $B$ and $C$ denote the set of partitions of $n$ with maximum part $17$ and also having $1$, $2$ and $3$ as a part, respectively. Then we need to find the cardinality of the set $A^{\complement} \cap B^{\complement} \cap C^{\complement}$. Using inclusion and exclusion principle, we get that the number of partitions of $n$ with all parts in the set $\{4,5, \ldots, 17\}$ is given by $p_{17}(n)-p_{17}(n-1)-p_{17}(n-2)+p_{17}(n-4)+p_{17}(n-5)-p_{17}(n-6)$, and thus can be easily computed.
\end{remark}

\section{The series $G_{L,s}(q)$ for $s \geq 4$ and other generalizations}

For fixed $L$ and $s$, let $p_{L,s}(q)$ be the
smallest degree polynomial in $q$ with smallest coefficients such that $G_{L,s}(q) +
p_{L, s}(q) \succeq 0$.   This paper finds $p_{L,s}(q)$ for $s=3$ and all $L$,
while the cases $s=2$ (found in Theorem \ref{s=2})
and $s=1$ (found in \cite{BerkovichAlexander2017SEPI}) were found earlier.  One goal is to find
$p_{L,s}(q)$ for general $s$ and $L$.  Our numerical
computations for $s=4$ and $s=5$ do not suggest a nice form for $p_{L,s}(q)$.
A simpler problem, though also interesting, is to determine the degree of
$p_{L,s}(q)$ for general $L$ and $s$.  Furthermore, we suspect that for fixed $s \geq 4$ that
$p_{L,s}(q)$ stabilizes;  that is, there is an $L_0$ such that $p_{L,s}(q) =
p_{L_0, s}(q)$ for all $L \geq L_0$, as in the case for $s=3$ (Theorem
\ref{GLthree} gives $L_0 = 10$ when $s=3$).  Finding $L_0$ as a function of $s$ is also
another open problem.

Moreover, series analogous to $G_{L,s}(q)$ for partitions with further
restrictions on parts, such as for partitions with only odd parts or for self
conjugate partitions, may prove interesting. 
%
%
%

\bibliography{GLs}
\bibliographystyle{amsalpha}

%

\end{document}